\documentclass{article}

\title{Open subgroups of free topological groups}
\author{Jeremy Brazas}
\usepackage{stmaryrd,graphicx,float}
\usepackage{amssymb,mathrsfs,amsmath,amscd,amsthm}
\usepackage[all,cmtip]{xy}
\DeclareMathAlphabet{\mathpzc}{OT1}{pzc}{m}{it}
\usepackage{amsfonts,txfonts,pxfonts,latexsym,wasysym}

\newcommand{\pyx}{p:Y\to X}
\newcommand{\ui}{[0,1]}

\newcommand{\piqtop}{\pi^{qtop}}

\newcommand{\fmx}{F_{M}(X)}

\newcommand{\fgx}{F_{G}(X,\ast)}
\newcommand{\fmy}{F_{M}(Y)}

\newcommand{\gamz}{\Gamma_{0}}
\newcommand{\gam}{\Gamma}
\newcommand{\gamast}{\Gamma^{\pm}}
\newcommand{\cgam}{\mathscr{C}(\gam)}
\newcommand{\fgam}{\mathscr{F}(\gam)}
\newcommand{\fgamv}{\mathscr{F}(\gam)(v)}
\newcommand{\topgraph}{\mathbf{TopGraph}}

\newcommand{\qtop}{\mathbf{qTop}}

\newcommand{\spaces}{\mathbf{Top}}

\newcommand{\bspaces}{\mathbf{Top_{\ast}}}

\newcommand{\pit}{\pi^{\tau}}

\newcommand{\topcat}{\mathbf{TopCat}}

\newcommand{\topgrp}{\mathbf{TopGrp}}

\newcommand{\tgam}{\widetilde{\gam}}

\newcommand{\stopol}{\mathbf{sTop}}
\newcommand{\mcc}{\mathcal{C}}
\newcommand{\mcd}{\mathscr{D}}
\newcommand{\mcg}{\mathcal{G}}
\newcommand{\scrc}{\mathscr{C}}

\newcommand{\mcp}{\mathcal{P}}

\newcommand{\mcu}{\mathcal{U}}

\newcommand{\mcv}{\mathcal{V}}

\newcommand{\scrp}{\mathscr{P}}
\newcommand{\mpgam}{\scrp_{\gam}}
\newcommand{\epgam}{\mathscr{E}_{\gam}}
\newtheorem{theorem}{Theorem}
\newtheorem{lemma}[theorem]{Lemma}

\newtheorem{corollary}[theorem]{Corollary}
\newtheorem{definition}[theorem]{Definition}

\newtheorem{example}[theorem]{Example}
\newtheorem{remark}[theorem]{Remark}

\newtheorem{freetopgroupoids}[theorem]{Free $\spaces$-groupoids}
\newtheorem{pathcomponentspaces}[theorem]{Path component spaces}

\newtheorem{freetopcat}[theorem]{Free $\spaces$-categories}

\begin{document}
\maketitle
\begin{abstract}
The theory of covering spaces is often used to prove the Nielsen-Schreier theorem, which states that every subgroup of a free group is free. We apply the more general theory of semicovering spaces to obtain analogous subgroup theorems for topological groups: Every open subgroup of a free Graev topological group is a free Graev topological group. An open subgroup of a free Markov topological group is a free Markov topological group if and only if it is disconnected.
\end{abstract}
\section{Introduction}
A well-known application of covering space theory is the Nielsen-Schreier theorem \cite{Schreier}, which states that every subgroup of a free group is free \cite{Brown06,Hatcher02}. The corresponding situation for topological groups is more complicated since it is not true that every closed subgroup of a free topological group is free topological \cite{Brown75,Clarke,Graev,HuntMorris}. The purpose of this paper is to use the theory of semicovering spaces developed in \cite{Brazsemi} to prove the following theorem.
\begin{theorem}\label{nielsonshreier}
Every open subgroup of a free Graev topological group is a free Graev topological group.
\end{theorem}

Free topological groups are important objects in the general theory of topological groups and have an extensive literature dating back to their introduction by A.A. Markov \cite{Markov} in the 1940s. Markov \cite{Markov} defined the free topological group $\fmx$ on a space $X$ and Graev \cite{Graev} later introduced the free topological group $\fgx$ on a space $X$ with basepoint $\ast\in X$. While the existence of these groups (for any $X$) follows abstractly from adjoint existence theorems \cite{Po91}, the theory of free topological groups has traditionally required the condition that $X$ be completely regular since the canonical injections $\sigma:X\to \fmx$ and $\sigma_{\ast}:X\to \fgx$ are embeddings if and only if $X$ is completely regular. In this paper, we exploit universal properties and adjoint functors (as opposed to working with complicated characterizations of the topology) to avoid placing any restrictions on $X$. For more on the theory of free topological groups, we refer the reader to \cite{AT,Sipacheva,Thomas}.

Topological versions of the Nielsen-Schreier Theorem \cite{BHsubgroup,Nickolas75} have been attempted for free topological groups on Hausdorff $k_{\omega}$-spaces, i.e. spaces which are the inductive limit of a sequence of compact subspaces. In this case, a subgroup of a free Graev topological group which admits a continuous Schreier traversal is free (Graev) topological. Theorem \ref{nielsonshreier} is an improvement in the sense that it holds for the free topological group on an arbitrary topological space. Another difference between our approach and that in \cite{BHsubgroup} is that we use topologically enriched graphs and categories as opposed to graphs and categories internal to a category of spaces.

Covering theoretic proofs of the algebraic Nielsen-Schreier Theorem typically require an understanding of covering spaces and fundamental group(oid)s of graphs. Our proof of Theorem \ref{nielsonshreier} generalizes this approach by replacing covering theory with the theory of semicoverings \cite{Brazsemi}, graphs with $\spaces$-graphs (i.e. topological graphs with discrete vertex spaces), and the fundamental groupoid (fundamental group) with the fundamental $\spaces$-groupoid \cite{Brazsemi} (topological fundamental group \cite{Braztopgrp}). Our application of the classification of semicoverings relies heavily on the fact that the theory applies to certain non-locally path connected spaces, called locally wep-connected spaces, which are not included in classical covering space theory.

This paper is structured as follows. In Section \ref{sectionftg}, we recall the basic theory of free topological groups and include a general comparison of the two notions of free topological groups (those in the sense of Graev and those in the sense of Markov). Using Theorem \ref{nielsonshreier}, we obtain a structure theorem for open subgroups of free Markov topological groups (Theorem \ref{mainresult2}): \textit{An open subgroup of a free Markov topological group is a free Markov topological group if and only if it is disconnected}.

In Section \ref{sectiongraphsandcats}, we extend the usual notion of an algebraic graph by allowing edge spaces to have non-discrete topologies; the resulting objects are called $\spaces$-graphs. We then present some universal constructions of topologically enriched categories and groupoids to be used in the computations of Section 4. In Section \ref{sectionfundgroupoid}, we show the fundamental $\spaces$-groupoid (resp. topological fundamental group) of a $\spaces$-graph is a free $\spaces$-groupoid (resp. free Graev topological group). Finally, in Section \ref{sectionsemicovering}, semicovering theory is applied to $\spaces$-graphs. Analogous to the fact that a covering of a graph is a graph, we find that a semicovering of a $\spaces$-graph is a $\spaces$-graph. The paper concludes with a proof of Theorem \ref{nielsonshreier}.
\section{Free topological groups}\label{sectionftg}
\begin{definition}\emph{
Let $X$ be a topological space. The \textit{free Markov topological group on} $X$ is the unique (up to isomorphism) topological group $\fmx$ equipped with a map $\sigma:X\to \fmx$ universal in the sense that every map $f:X\to G$ to a topological group $G$ induces a unique, continuous homomorphism $\hat{f}:\fmx\to G$ such that $\hat{f}\sigma=f$.}
\end{definition}

The existence of free Markov topological groups is guaranteed by the General Adjoint Theorem \cite{Po91}. In particular, if $\spaces$ is the category of topological spaces and $\topgrp$ is the category of topological groups, then $F_M:\spaces\to \topgrp$ is left adjoint to the forgetful functor $\topgrp\to \spaces$. Algebraically, $\fmx$ is the free group on the underlying set of $X$ and $\sigma:X\to \fmx$ is the canonical injection of generators.
\begin{definition} \emph{
Let $X$ be a space with basepoint $\ast\in X$. The \textit{free Graev topological group} on $(X,\ast)$ is the unique (up to isomorphism) topological group $F_G(X,\ast)$ equipped with a map $\sigma_{\ast}:X\to \fgx$ such that $\sigma_{\ast}(\ast)$ is the identity element of $\fgx$ and universal in the sense that every map $f:X\to G$ to a topological group $G$ which takes $\ast$ to the identity element of $G$ induces a unique, continuous homomorphism $\tilde{f}:\fgx\to G$ such that $\tilde{f}\sigma_{\ast}=f$.}
\end{definition}
Similar to the unbased case, if $\bspaces$ is the category of based topological spaces, then $F_G:\bspaces\to \topgrp$ is left adjoint to the forgetful functor $\topgrp\to\bspaces$. Here, the basepoint of a topological group is the identity element. Algebraically, $\fgx$ is the free group on the set $X\backslash \{\ast\}$, however, it is not necessarily isomorphic to $F_M(X\backslash \{\ast\})$ as a topological group. On the other hand, $\fmx$ is isomorphic to the free Graev topological group $F_G(X_+,\ast)$ where $X_+=X\sqcup\{\ast\}$ has an isolated basepoint and $\fgx$ is isomorphic to the quotient topological group $\fmx/N$ where $N$ is the conjugate closure of $\{\ast\}$.

Graev showed in \cite[Theorem 2]{Graev} that the isomorphism class of $\fgx$ as a topological group does not depend on the choice of basepoint, i.e. given any other point $\ast '\in X$ there is an isomorphism $F_{G}(X,\ast)\to F_{G}(X,\ast ')$ of topological groups.\footnote{Graev assumes $X$ is completely regular, however, the argument given also applies to the general case.} It remains to understand when $\fgx$ is isomorphic to the free Markov topological group $\fmy$ on some space $Y$. This is answered by Graev in \cite{Graev} in the case that $X$ is completely regular; we modify Graev's argument only slightly in order to obtain a more general result.
\begin{lemma}\label{connected}
If $X$ is the disjoint union $X=A_1\sqcup A_2$ of open sets $A_i\subset X$ and $e_i\in A_i$, then $F_G(X,e_1)$ is isomorphic to the free Markov topological group $F_M(A_1\vee A_2)$ on the wedge sum $A_1\vee A_2=X/\{e_1,e_2\}$.
\end{lemma}
\begin{proof}
Let $q:X\to A_1\vee A_2$ be the quotient map making the identification $q(e_1)=z=q(e_2)$. Define a map $f:X\to \fmx$ by $f(a)=ae_2^{-1}$ for $a\in A_1$ (here $a_1e_{2}^{-1}$ is the product in $\fmx$) and $f(a)=a$ for $a\in A_2$. Since $F_M(q)(e_1)=F_{M}(q)(e_2)$, the composition $\psi=F_{M}(q)f:X\to F_{M}(A_1\vee A_2)$ takes $e_1$ to the identity of $F_{M}(A_1\vee A_2)$ and induces a continuous homomorphism $\tilde{\psi}:F_G(X,e_1)\to F_{M}(A_1\vee A_2)$. Note that $\tilde{\psi}(e_2)=z$.

Now consider the map $g:X\to F_G(X,e_1)$ where $g(a)=ae_2$, $a\in A_1$ is the product in $F_G(X,e_1)$ and $g(a)=a$, $a\in A_2$. Since $e_1$ is the identity in $F_{G}(X,e_1)$, $g(e_1)=e_1e_2=e_2=g(e_2)$. We obtain a continuous map $\phi:A_1\vee A_2\to F_G(X,e_1)$ on the quotient such that $\phi(z)=e_2$ and which induces a continuous homomorphism $\hat{\phi}:F_{M}(A_1\vee A_2)\to F_G(X,e_1)$.

A direct check shows that $\hat{\phi} \tilde{\psi}$ is the identity homomorphism of $F_{G}(X,e_1)$ and $\tilde{\psi}\hat{\phi}$ is the identity of $F_{M}(A_1\vee A_2)$. In particular, if $a\in A_1\backslash\{e_1\}$, then $\hat{\phi}\tilde{\psi}(a)=\hat{\phi}(ae_{2}^{-1})= \hat{\phi}(a)\hat{\phi}(e_2)^{-1}=(ae_2)e_{2}^{-1}=a$ and $\tilde{\psi}\hat{\phi}(a)=\tilde{\psi}(ae_2)=\tilde{\psi}(a)\tilde{\psi}(e_2) =ae_{2}^{-1}e_2=a$. The other cases are straightforward and left to the reader.
\end{proof}
\begin{theorem}\label{connected2}
For any space $X$, the following are equivalent.
\begin{enumerate}
\item $X$ is connected.
\item $\fgx$ is connected.
\item $\fgx$ is not isomorphic to a free Markov topological group.
\end{enumerate}
\end{theorem}
\begin{proof}
1. $\Rightarrow$ 2. Suppose $X$ is connected and let $C$ be the connected component of the identity in $\fgx$. Since $\sigma_{\ast}:X\to \fgx$ is continuous, the generating set $\sigma_{\ast}(X)$ is a connected subspace of $\fgx$ containing the identity and is therefore contained in $C$. In a general topological group, the connected component of the identity element is a subgroup \cite[1.4.26]{AT}. Therefore $C=\fgx$.\\
2. $\Rightarrow$ 3. Every free Markov topological group is disconnected since the canonical map $X\to \ast$ collapsing $X$ to a point induces a continuous homomorphism $\fmx\to F_M(\ast)=\mathbb{Z}$ onto to the discrete group of integers. Therefore, if $\fgx$ is connected, $\fgx$ cannot be isomorphic to a free Markov topological group.\\
3. $\Rightarrow$ 1. This follows directly from Lemma \ref{connected}.
\end{proof}
Combining Theorems \ref{nielsonshreier} and \ref{connected2} and the fact that every free Markov topological group is a free Graev topological group, we obtain a structure theorem for open subgroups of free Markov topological groups. This result generalizes that in \cite{BHsubgroup} for free Markov topological groups on Hausdorff $k_{\omega}$-spaces.
\begin{theorem}\label{mainresult2}
An open subgroup of a free Markov topological group is a free Markov topological group if and only if it is disconnected.
\end{theorem}
\section{Topologically enriched graphs and categories}\label{sectiongraphsandcats}
The rest of this paper is devoted to a proof of Theorem \ref{nielsonshreier}.
\subsection{$\spaces$-graphs}
A $\spaces$-\textit{graph} $\gam$ consists of a discrete space of vertices $\gam_0$, an edge space $\gam$, and continuous structure maps $\partial_{0},\partial_{1}:\gam\to \gam_0$. For convenience, we sometimes let $\gam$ denote the $\spaces$-graph itself. The set of composable edges in $\gam$ is the pullback $\gam\times_{\gamz}\gam=\{(e,e')|\partial_1(e)=\partial_0(e')\}$.

For each pair of vertices $x,y\in \gam_0$, let $\gam_x=\partial_{0}^{-1}(x)$, $\gam^y=\partial_{1}^{-1}(y)$, and $\gam(x,y)=\gam_x\cap \gam^y$. Since we require the vertex space of a $\spaces$-graph to be discrete, the edge space decomposes as the topological sum $\gam=\coprod_{(x,y)\in X\times X}\gam(x,y)$ over ordered pairs of vertices.

Since it is possible that both $\gam(x,y)$ and $\gam(y,x)$ are non-empty, we are motivated to make the following construction. Let $\gam(x,y)^{-1}$ denote a homeomorphic copy of $\gam(x,y)$ for each pair $(x,y)\in \gamz\times \gamz$. Here $e\in \gam(x,y)$ corresponds to $e^{-1}\in \gam(x,y)^{-1}$. Define a new $\spaces$-graph $\gamast$ to have vertex spaces $\gam_0$ and $\gamast(x,y)=\gam(x,y)\sqcup \gam(y,x)^{-1}$. In particular, note that $\gamast(x,x)=\gam(x,x)\sqcup \gam(x,x)^{-1}$.

A morphism $f:\gam\to \gam '$ of $\spaces$-graphs consists of a pair of continuous functions $(f_0,f):(\gam_0,\gam)\to (\gam_{0}^{\prime},\gam^{\prime})$  such that $\partial_{i}'\circ f=f_0\circ \partial_i$, $i=1,2$. Such a morphism is said to be \textit{quotient} if $f_0$ and $f$ are quotient maps of spaces (note $f_0$ only needs to be surjective to be quotient). There is also an obvious notion of sub-$\spaces$-graph $S\subseteq \gam$. We say such a sub-$\spaces$-graph is \textit{wide} if $S_0=\gam_0$. The category of $\spaces$-graphs is denoted $\topgraph$.

\begin{definition}
\emph{The \textit{geometric realization} of a $\spaces$-graph $\gam$ is the topological space \[|\gam|=\gam_0\sqcup( \gam\times [0,1])/\sim\text{ where }\partial_i(\alpha)\sim (\alpha,i)\text{ for }i=0,1.\] A $\spaces$-graph $\gam$ is \textit{connected} if $|\gam|$ is path connected, or equivalently, if for each $x,y\in \gam_0$, there is a sequence of vertices $x=a_1,a_2,...,a_n=y$ such that $\gam^{\pm}(a_j,a_{j+1})\neq \emptyset$ for $j=1,...,n-1$.
}
\end{definition}
We typically assume $\spaces$-graphs are connected.
\begin{remark}\emph{
For any $0<r<1$, the image of $\gam_x\times [0,r]$ in the quotient $|\gam|$ is homeomorphic to the cone $CX$ on $X$. Similarly, if $Z=\gam(x,y)\sqcup \gam(y,x)$, then the image of $Z \times [0,1]$ in the $|\gam|$ is the unreduced suspension $SZ$. Finally, note \[|\gam|\backslash\gam_0=\coprod_{(x,y)}\left(\gam(x,y)\times (0,1)\right).\]}
\end{remark}
 \begin{figure}[H] \centering \includegraphics[height=2in]{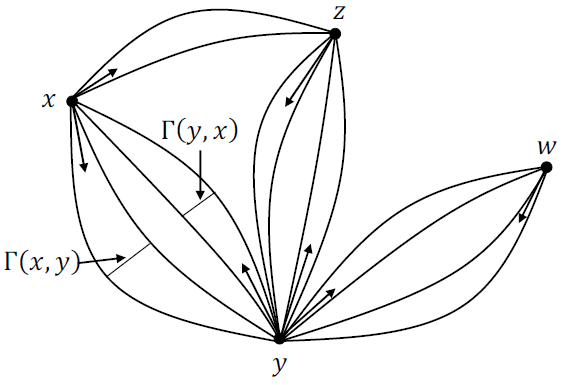}
\caption{The realization of a $\spaces$-graph $\gam$ with four vertices. Here $\gam(z,x)$, $\gam(w,z)$, and $\gam(z,w)$ are all empty.}
\end{figure}
\begin{remark}\label{basisgam}\emph{
It is unfortunate that $|\gam|$ need not be first countable at its vertices, however, it is possible to change the topology on $|\gam|$ without changing homotopy type so that each vertex has a countable neighborhood base. The \textit{vertex neighborhood of} $x\in \gam_0$ of radius $r\in (0,1)$ is the image of\[\left(\gam_x\times [0,r)\right) \cup \left(\gam^x\times (1-r,1]\right)\]and is denoted $B(x,r)$. An \textit{edge neighborhood} of a point $(e,t)\in \gam(x,y)\times (0,1)$ is the homeomorphic image of a set $U\times (a,b)$ where $U$ is an open neighborhood of $e$ in $\gam(x,y)$ and $0<a<t<b<1$. The basis consisting of vertex and edge neighborhoods is closed under finite intersection and generates a topology which may be strictly coarser than the quotient topology (but only at vertices). Note that each vertex neighborhood $B(x,r)$ is contractible onto $x$ and the set $\{B(x,1/n)|n\geq 1\}$ is a countable neighborhood base at $x$.
}
\end{remark}
From now on, we assume $|\gam|$ has the courser topology generated by vertex and edge neighborhoods.
\begin{example}
\emph{
If $\gam$ is a $\spaces$-graph with a single vertex, then $|\gam|$ is the \textit{generalized wedge of circles} $\Sigma(\gam_+)$ (where $\Sigma$ denotes reduced suspension) studied in detail in \cite{Brazfretopgrp}. When $\gam_0=\{x_0,x_1\}$ and the structure maps are the two constant maps $\partial_i:\gam\to \gam_0$, $\partial_{i}(\alpha)=x_i$ (equivalently, $\gam=\gam(x_0,x_1)$), then $|\gam|$ is the unreduced suspension $S\gam$. Thus, unlike a discrete graph, a $\spaces$-graph may be simply connected but not contractible, e.g. $\gam=S^1$.
}
\end{example}
\begin{pathcomponentspaces}
\emph{
We end this section with a useful construction on $\spaces$-graphs: The \textit{path component space} of a topological space $X$ is the quotient space $\pi_0 (X)$ where each path component is identified to a point. If $\gam$ is a $\spaces$-graph, then $\pi_0(\gam)$ is the $\spaces$-graph with vertex space $\gam_0$ and $\pi_0(\gam)(x,y)=\pi_0(\gam(x,y))$. The canonical quotient morphism of $\spaces$-graphs $q:\gam\to \pi_0(\gam)$ is the identity on vertices and takes an edge $e$ to its path component $[e]$.}
\end{pathcomponentspaces}
\begin{remark}\label{everygraphisapcgraph}
\emph{
Another useful construction is a section to the path component functor $\pi_0:\topgraph\to\topgraph$. Given any space $X$, there a (paracompact Hausdorff) space $h(X)$ and a natural homeomorphism $\pi_0(h(X))\cong X$ \cite{Har80}. Thus for any $\spaces$-graph $\gam$, we define $h(\gam)$ to have object space $\gam_0$ and $h(\gam)(x,y)=h(\gam(x,y))$ so that $\pi_0(h(\gam))\cong \gam$.
}\end{remark}
\subsection{$\spaces$-categories and $\qtop$-categories}
Our use of enriched categories aligns with that in \cite{KellyEnriched}. If a $\spaces$-graph $\mcc$ comes equipped with continuous composition map $\mcc\times_{\mcc_0}\mcc\to \mcc$ making $\mcc$ a category in the usual way, then $\mcc$ is a $\spaces$\textit{-category} (or a category enriched over $\spaces$). Since $Ob(\mcc)=\mcc_0$ is discrete, $\mcc\times_{\mcc_0}\mcc$ decomposes as a topological sum of products $\mcc(x,y)\times\mcc(y,z)$. Thus to specify a $\spaces$-category one only need specify the hom-spaces $\mcc(x,y)$ and continuous composition maps $\mcc(x,y)\times\mcc(y,z)\to \mcc(x,z)$. If composition maps are only continuous in each variable, then $\mcc$ is an $\stopol$-category (the "s" is for "semitopological" as in \cite{AT}). A $\spaces$-functor $F:\mcc\to\mathcal{D}$ of $\spaces$-categories is a functor such that each function $F_{x,y}:\mcc(x,y)\to \mathcal{D}(F(x),F(y))$ is continuous. The category of $\spaces$-categories and $\spaces$-functors is denoted $\topcat$.

An \textit{involution} on a small category $\mcc$ is a function $\mcc\to \mcc$ defined by functions $\mcc(x,y)\to \mcc(y,x)$, $f\mapsto f^{\ast}$ such that $(f^{\ast})^{\ast}=f$, $(fg)^{\ast}=g^{\ast}f^{\ast}$, and $(id_{x})^{\ast}=id_x$. A $\spaces$-category (resp. a $\stopol$-category) equipped with a continuous involution is a $\spaces$\textit{-category with continuous involution} (resp. a $\qtop$\textit{-category}). If $\mcg$ is a $\spaces$-category (resp. $\stopol$-category) whose underlying category is a groupoid and the involution given by the inversion functions $\mcg(x,y)\to \mcg(y,x)$ is continuous, then $\mcg$ is a $\spaces$\textit{-groupoid} (resp. $\qtop$-groupoid).

A functor $F:\mcc\to\mathcal{D}$ of categories with involution \textit{preserves involution} if $F(f^{\ast})=F(f)^{\ast}$. In particular, a $\qtop$-functor $F:\mcc\to \mathcal{D}$ of $\qtop$-categories is an involution preserving functor which is continuous on hom-spaces.

The notion of $\qtop$-groupoid is particularly relevant and is studied in Section 4 of \cite{Brazsemi}. The following Lemma is a useful fact asserting that the category $\mathbf{TopGrpd}$ of $\spaces$-groupoids is a full reflective subcategory of the category $\mathbf{qTopGrpd}$ of $\qtop$-groupoids.

\begin{lemma}\label{reflection}
\cite[Lemma 4.5]{Brazsemi} The forgetful functor $\mathbf{TopGrpd}\to \mathbf{qTopGrpd}$ has a left adjoint $\tau:\mathbf{qTopGrpd}\to \mathbf{TopGrpd}$ which is the identity on the underlying groupoids and functors.
\end{lemma}
\begin{freetopcat}
\emph{
The \textit{free} $\spaces$\textit{-category} generated by a $\spaces$-graph $\gam$ is the $\spaces$-category $\cgam$ with object space $\gam_0$ and in which morphisms are finite sequences $e_{1}e_{2}...e_{n}$ of composable edges $e_i\in\gam$. In particular, the hom-space $\cgam(x,y)$ is topologized as the topological sum $\coprod \gam(x,a_1)\times \gam(a_1,a_2)\times \dots \times\gam(a_n,y)$ where the sum ranges over all finite sequences $a_1,...,a_n$ in $\gam_0$. In order to obtain a category, we add an isolated identity morphism $\{id_x\}$ to each space $\cgam(x,x)$. Note this construction yields a functor $\scrc:\topgraph\to \topcat$ left adjoint to the forgetful functor $\topcat\to \topgraph$.\\
\indent The construction of $\cgam$ is easily modified to include a continuous involution. In particular, the free $\spaces$-category with (continuous) involution on $\gam$ is $\scrc^{\pm}(\gam)=\scrc(\gamast)$, the free-$\spaces$ category on the $\spaces$-graph $\gamast$ described in the previous section. Thus a generic (non-identity) morphism of $\scrc^{\pm}(\gam)$ may be given by a sequence $e_{1}^{\delta_1}e_{2}^{\delta_2}...e_{n}^{\delta_n}$ where $e_i\in \gam$, $\delta_i\in \{\pm 1\}$.
}
\end{freetopcat}
\begin{remark}
\emph{
The construction of the path component $\spaces$-graph $\pi_0(\gam)$ also applies to $\spaces$-categories. Recall that for spaces $X,Y$, there is a canonical, continuous bijection $\psi:\pi_0(X\times Y)\to \pi_0(X)\times \pi_0(Y)$ which is not necessarily a homeomorphism \cite{Brazfretopgrp}. Consequently, if $\gam$ is a $\spaces$-category (with continuous involution), then $\pi_0(\gam)$ naturally inherits the structure of a $\stopol$-category ($\qtop$-category) but is not always a $\spaces$-category (with continuous involution). While it is possible to avoid this difficulty by restricting to a cartesian closed category of spaces, we remain in the usual topological category in order to prove Theorem \ref{nielsonshreier} in full generality.
}
\end{remark}
The observation on products in the previous remark immediately extends to the following lemma.
\begin{lemma}\label{canonicalfunctor}
Given a $\spaces$-graph $\gam$, there is a canonical $\qtop$-functor $\psi:\pi_0(\scrc^{\pm}(\gam))\to \scrc^{\pm}(\pi_0(\gam))$ given by $[e_{1}^{\delta_1}e_{2}^{\delta_2}...e_{n}^{\delta_n}]\mapsto [e_1]^{\delta_1}[e_2]^{\delta_2}...[e_n]^{\delta_n}$ which is an isomorphism of the underlying categories.
\end{lemma}
\begin{freetopgroupoids} \label{freetopgroupoids}
\emph{
Given a $\spaces$-graph $\gam$, the free $\spaces$-groupoid generated by $\gam$ is denoted $\mathscr{F}(\gam)$ and is characterized by the following universal property: If $\mathcal{G}$ is a $\spaces$-groupoid any $\spaces$-graph morphism $f:\gam\to \mathcal{G}$ extends uniquely to a $\spaces$-functor $\hat{f}:\mathscr{F}(\gam)\to \mathcal{G}$. In other words, $\mathscr{F}:\topgraph\to\mathbf{TopGrpd}$ is left adjoint to the forgetful functor $\mathbf{TopGrpd}\to \topgraph$.\\
\indent The underlying groupoid of $\mathscr{F}(\gam)$ is simply the free groupoid generated by the underlying algebraic graph of $\gam$, i.e. $Ob(\mathscr{F}(\gam))=\gamz$ and a morphism is a reduced word $e_{1}^{\delta_1}e_{2}^{\delta_2}...e_{n}^{\delta_n}\in \scrc^{\pm}(\gam)$. See \cite{Brown06,Hig71} for more on free groupoids. The topological structure of $\mathscr{F}(\gam)$ is characterized as follows: Let $\mathscr{F}_{R}(\gam)$ be the free groupoid on the underlying algebraic graph of $\gam$ which is the quotient of $\scrc^{\pm}(\gam)$ with respect to the word reduction functor $R:\scrc^{\pm}(\gam)\to \mathscr{F}_{R}(\gam)$. Note that $\mathscr{F}_{R}(\gam)$ is a $\qtop$-groupoid. The free $\spaces$-groupoid is the $\tau$-reflection \[\mathscr{F}(\gam)=\tau\left(\mathscr{F}_{R}(\gam)\right).\]
It is straightforward to verify that this groupoid has the desired universal property. In the case that $\gam$ has a single vertex \cite{Brazfretopgrp}, $\mathscr{C}^{\pm}(\gam)$ is the free topological monoid with continuous involution on $\gam$ and $\mathscr{F}(\gam)$ is the free topological group $F_M(\gam)= F_G(\gam_+,\ast)$ .
}
\end{freetopgroupoids}
\subsection{Vertex groups and free topological groups}\label{vertexgroups}
We now show each vertex group of a free $\spaces$-groupoid is a free Graev topological group.
\begin{definition}
\emph{
A $\spaces$-graph $T$ is a \textit{tree} if $T$ is discrete and $|T|$ is contractible. If $\gam$ is a $\spaces$-graph, a tree $T\subseteq \gam$ is \textit{maximal} in $\gam$ if $T_0=\gam_0$. Even though $T$ is itself a discrete $\spaces$-graph, the edge space $T$ need not be open in $\gam$. A \emph{tree groupoid} is a groupoid $\mcg$ such that each set $\mcg(x,y)$ has exactly one element. Note that if $T$ is a tree, then $\mathscr{F}(T)$ is a discrete tree groupoid.}
\end{definition}
The standard argument that every graph contains a maximal tree is the same for $\spaces$-graphs.
\begin{lemma}
Every connected $\spaces$-graph contains a maximal tree.
\end{lemma}
Fix a $\spaces$-graph $\gam$, a maximal tree $T\subset \gam$, and a vertex $v\in \gamz$. Let $\mathscr{F}(\gam)(v)=\mathscr{F}(\gam)(v,v)$ be the vertex topological group at $v$. Recall that if $\gam$ has a single vertex $v$, then $\mathscr{F}(\gam)=\mathscr{F}(\gam)(v)\cong F_{M}(\gam)\cong F_{G}(\gam_+,\ast)$. Therefore, we restrict to the case when $\gam$ has more than one vertex. In this case, $T$ has non-empty edge space.

 For vertex $x\in \gamz$, let $\gamma_{v,x}$ be the unique element of $\mathscr{F}(T)(v,x)$. Define a retraction $r_T:\mathscr{F}(\gam)\to \mathscr{F}(\gam)(v)$ of groupoids so that if $\alpha\in \mathscr{F}(\gam)(x,y)$, then $r_T(\alpha)=\gamma_{v,x}\alpha\gamma_{y,v}$. By definition, if $i:\mathscr{F}(\gam)(v)\to \mathscr{F}(\gam)$ is the inclusion of the vertex group, then $r_Ti$ is the identity of $\mathscr{F}(\gam)(v)$. In fact, since composition in $\mathscr{F}(\gam)$ is continuous, we have a $\spaces$-functor.
\begin{lemma}
$r_T:\mathscr{F}(\gam)\to \mathscr{F}(\gam)(v)$ is a retraction of $\spaces$-groupoids.
\end{lemma}
Let $\sigma:\gam\to\mathscr{F}(\gam)$ be the canonical $\spaces$-graph morphism. It is known that the underlying group of $\mathscr{F}(\gam)(v)$ is freely generated by the set $r_T\sigma(\gam\backslash T)$ (See, for instance, \cite[8.2.3]{Brown06}).

Note that if $\gamma\in \gam$, then $r_T\sigma(\gamma)$ is the identity element of the group $\mathscr{F}(\gam)(v)$ if and only if $\gamma\in T$. Let $\gam/T$ be the quotient of the edge space $\gam$ and choose the basepoint $\ast\in\gam/T$ to be the the image of $T$. Thus the function $r_T\sigma:\gam\to \mathscr{F}(\gam)(v)$ induces a continuous injection $s:\gam/T\to \mathscr{F}(\gam)(v)$ such that $sq=r_T\sigma$ and where $s(\ast)$ is the identity element of $\mathscr{F}(\gam)$. Since $r_T\sigma(\gam\backslash T)$ freely generates $\mathscr{F}(\gam)(v)$, the continuous group homomorphism $\tilde{s}:F_G(\gam/T,\ast)\to \mathscr{F}(\gam)(v)$ induced by $s$ is an isomorphism of groups.
\begin{theorem}\label{vertexgroup}
If $\gam$ has more than one vertex and $T\subset \gam$ is a maximal tree, then the vertex group $\mathscr{F}(\gam)(v)$ is isomorphic to the free Graev topological group $F_G(\gam/T,\ast)$.
\end{theorem}
\begin{proof}
It suffices to show the inverse of $\tilde{s}:F_G(\gam/T,\ast)\to \mathscr{F}(\gam)$ is continuous. Let $\sigma_{\ast}:\gam/T\to F_G(\gam/T,\ast)$ be the inclusion of generators and $q:\gam\to \gam/T$ be the quotient map. The composition $g=\sigma_{\ast}q:\gam\to F_G(\gam/T,\ast)$ may be viewed as a morphism of $\spaces$-graphs taking all vertices of $\gam$ to the unique vertex of $F_G(\gam/T,\ast)$. Since $F_G(\gam/T,\ast)$ is a $\spaces$-groupoid, there is a unique $\spaces$-functor $\hat{g}:\mathscr{F}(\gam)\to F_G(\gam/T,\ast)$ such that $\hat{g}\sigma=g$. If $i:\mathscr{F}(\gam)(v)\to \mathscr{F}(\gam)$ is the inclusion of the vertex group, the composition $\hat{g}i:\mathscr{F}(\gam)(v)\to F_G(\gam/T,\ast)$ is continuous. We check that $\hat{g}i$ is the inverse of $\tilde{s}$.

Recall that $sq=r_T\sigma$. Therefore \[\tilde{s}\hat{g}\sigma=\tilde{s}g=\tilde{s}\sigma_{\ast}q=sq=r_T\sigma.\]Uniqueness of extensions then gives $\tilde{s}\hat{g}=r_T$.
\[\xymatrix{
& \gam \ar[dl]_-{\sigma} \ar[dr]^-{q} \ar[d]^{g} & \\
\fgam \ar[r]^-{\hat{g}}  \ar@<.4ex>[dr]^-{r_T} & F_G(\gam/T,\ast) \ar[d]^-{\tilde{s}} & \gam/T \ar[l]_-{\sigma_{\ast}} \ar[dl]^-{s}\\
& \fgamv \ar@<.4ex>[ul]^-{i}
}\]
It is now clear that $\tilde{s}\hat{g}i=r_Ti=id$. Finally, since $\tilde{s}\hat{g}i\tilde{s}=r_Ti\tilde{s}=\tilde{s}$ and $\tilde{s}$ is injective, we have $\hat{g}i\tilde{s}=id$.
\end{proof}
\section{The fundamental $\spaces$-groupoid of a $\spaces$-graph}\label{sectionfundgroupoid}
\subsection{Path spaces and the fundamental $\spaces$-groupoid}

For a given space $X$, let $\mcp X$ be the space of paths $\ui\to X$ with the compact-open topology generated by subbasis sets $\langle C,W\rangle =\{\alpha|\alpha(C)\subseteq W\}$ where $C\subseteq \ui$ is compact and $W\subseteq X$ is open. For a closed subinterval $K\subseteq \ui$, let $L_K\colon \ui\to K$ be the unique, increasing, linear homeomorphism. If $\alpha:\ui\to X$ is a path, let $\alpha_{K}=\alpha|_{K}\circ L_{K}$ be the restricted path of $\alpha$ to $K$. If $K=\{t\}\subseteq \ui$, take $\alpha_{K}$ to be the constant path $c_{\alpha(t)}$ at $\alpha(t)$. The concatenation $\alpha=\alpha_1\cdot\alpha_2\cdot ...\cdot \alpha_n$ of paths $\alpha_j$ such that $\alpha_{j+1}(0)=\alpha_{j}(1)$ is given by letting $\alpha_{\left[\frac{j-1}{n},\frac{j}{n}\right]}=\alpha_j$.

Consider a basic open neighborhood $\mcu=\bigcap_{j=1}^{n}\langle C_j,U_j\rangle$ of a path $\alpha$ and any closed interval $K\subseteq\ui$. Then $\mcu_{K}=\bigcap_{K\cap C_j\neq \emptyset}\langle L_{K}^{-1}(K\cap C_j),U_j\rangle$ is an open neighborhood of $\alpha_{K}$. If $K=\{t\}$ is a singleton, let $\mcu_{K}=\langle \ui,\bigcap_{t\in C_j}U_j\rangle$. On the other hand, if $\beta_{K}=\alpha$, then $\mcu^{K}=\bigcap_{j=1}^{n}\langle L_{K}(C_j),U_j\rangle$ is an open neighborhood of $\beta$. If $K=\{t\}$ so that $\alpha=c_{\alpha(t)}$, let $\mcu^{K}=\bigcap_{j=1}^{n}\langle \{t\},U_j\rangle$.
\begin{lemma} Let $\mcu=\bigcap_{j=1}^{n}\langle C_j,U_j\rangle$ be an open neighborhood in $\mcp X$ such that $\bigcup_{j=1}^{n}C_j=\ui$. Then
\begin{enumerate}
\item For any closed interval $K\subseteq \ui$, $(\mcu^{K})_{K}=\mcu\subseteq (\mcu_{K})^{K}$,
\item If $0=t_0\leq t_1\leq t_2\leq ...\leq t_n=1$, then $\mcu= \bigcap_{j=1}^{n}( \mcu_{[t_{j-1},t_j]})^{[t_{j-1},j_i]}.$
\end{enumerate}
\end{lemma}
In the case that $X$ is a $\spaces$-graph $|\gam|$, recall that vertex neighborhoods $B(x,r)$ and edge neighborhoods $U\times (a,b)$ (as in Remark \ref{basisgam}) form a basis $\mathscr{B}_{\gam}$ for the topology of $|\gam|$ which is closed under finite intersection. Thus sets of the form $\bigcap_{j=1}^{n}\left\langle \left[\frac{j-1}{n},\frac{j}{n}\right],U_j\right\rangle$, where $\{U_j\}\subset\mathscr{B}_{\gam}$, give basis generating the topology of $\mcp|\gam|$, which is convenient for our purposes. An open set of this form is said to be \textit{standard}.

The path space $\mcp|\gam|$ can be used to construct a $\spaces$-graph $\mpgam$ of paths: The object space of $\mpgam$ is the vertex set $\gamz$ and $\mpgam(x,y)$ is the subspace of $\mcp|\gam|$ consisting of paths from $x$ to $y$. Though concatenation of paths gives a continuous operation $\cdot:\mpgam(x,y)\times \mpgam(y,z)\to \mpgam(x,z)$, $\mpgam$ is not a $\spaces$-category because concatenation is not associative. One can obtain a $\spaces$-category by replacing paths with Moore paths, however, to remain consistent with \cite{Brazsemi}, we refrain from doing so.

We now recall the topologically enriched version of the usual fundamental groupoid used in \cite{Brazsemi} as applied to the case of $\spaces$-graphs.

\begin{definition}\label{deffundgroupoid}
\emph{
The \textit{fundamental} $\qtop$\textit{-groupoid} of a $\spaces$-graph $\gam$ is the $\qtop$-groupoid $\pi^{qtop}(\gam,\gamz)=\pi_0(\mpgam)$ whose object space is $\gamz$ and $\piqtop(\gam,\gamz)(x,y)$ is the path component space $\pi_0(\mpgam(x,y))$. The canonical quotient morphism is denoted $\pi:\mpgam\to \pi^{qtop}(\gam,\gamz)$. The \textit{fundamental} $\spaces$\textit{-groupoid} of $\gam$ is the $\tau$-reflection $\pit(\gam,\gamz)=\tau(\piqtop(\gam,\gamz))$.
}
\end{definition}
The underlying groupoid of $\pi^{qtop}(\gam,\gamz)$ and $\pit(\gam,\gamz)$ is the familiar fundamental groupoid $\pi(|\gam|,\gamz)$ with set of basepoints $|\gam|$ \cite{Brown06}.
\subsection{$\pit(\gam,\gamz)$ is a free $\spaces$-groupoid}
\begin{definition}
\emph{
A path $\alpha:\ui\to |\gam|$ is an \textit{edge path} if $\alpha^{-1}(\gam_0)=\{0,1\}$. An edge path $\alpha$ is \textit{trivial} if it is a null-homotopic loop. Equivalently, an edge path is non-trivial if and only if the endpoints are distinct or if it traverses a generalized wedge of circles $\Sigma(\gam(x,x)_+)\subset |\gam|$ at a vertex $x$. Let $\epgam$ be the wide sub-$\spaces$-graph of $\mpgam$ consisting of non-trivial edge paths.
}
\end{definition}
The following Lemma is a straightforward application of the existence of Lebesgue numbers.
\begin{lemma}\label{standardnbhdofedgepath}
If $\mcv$ is an open neighborhood of an edge path $\alpha$ in $\epgam(x,y)$, then there is a standard neighborhood $ \mathcal{A}=\bigcap_{j=1}^{n}\left\langle \left[\frac{j-1}{n},\frac{j}{n}\right],U_j\right\rangle$, $n>2$ of $\alpha$ in $\mcp|\gam|$ such that $\mathcal{A}\cap \epgam(x,y)\subseteq \mcv$ and such that $U_1,U_n$ are vertex neighborhoods and $U_2,...,U_{n-1}$ are edge neighborhoods.
\end{lemma}
Note for each edge $e\in \gam(x,y)$, there is a canonical non-trivial edge path $\alpha_e:\ui\to |\gam|$ where $\alpha_e(t)$ is the image of $(e,t)\in \gam(x,y)\times [0,1]$ in $|\gam|$.
\begin{lemma}\label{iso1}
There is a canonical embedding $\gamast\to \epgam$ of $\spaces$-graphs which induces an isomorphism $\pi_{0}(\scrc^{\pm}(\gam))=\pi_{0}(\scrc(\gamast))\to \pi_{0}(\scrc(\epgam))$ of $\qtop$-categories.
\end{lemma}
\begin{proof}
The embedding is the identity on objects and $e\mapsto \alpha_{e}$ for $e\in \gam(x,y)$ and $e^{-1}\mapsto \overline{\alpha_{e}}$ for $e^{-1}\in \gam(y,x)^{-1}$. Here $\overline{\beta}(t)=\beta(1-t)$ denotes the reverse path of $\beta$.

Define a $\spaces$-graph morphism $eval:\epgam\to \gamast$ as follows: It is the identity on vertices. If $\alpha\in \epgam(x,y)$, then $\alpha(1/2)$ is the image of a point $(g,s)\in \left(\gam(x,y)\sqcup \gam(y,x)\right)\times (0,1)$ in $|\gam|$. Take $eval(\alpha)=g$. It is straightforward to check that $eval$ induces the inverse $\qtop$-functor $\pi_{0}(\scrc(\epgam))\to \pi_{0}(\scrc(\gamast))$.
\end{proof}

Since concatenation $(\alpha,\beta)\mapsto \alpha\cdot \beta$ of paths is continuous, the inclusion $\epgam\to\mpgam$ gives rise to a $\spaces$-graph morphism $\scrc(\epgam)\to \mpgam$, $\alpha_1\alpha_2...\alpha_n\to \alpha_1\cdot \alpha_2\cdot ...\cdot\alpha_n$ on the free $\spaces$-category. Application of the path component space functor gives a $\qtop$-functor $\phi:\pi_0(\scrc(\epgam))\to \pi_0(\mpgam)=\piqtop(\gam,\gamz)$ to the fundamental $\qtop$-groupoid.
\begin{lemma} \label{quotient1}
The $\qtop$-functor $\phi:\pi_0(\scrc(\epgam))\to \piqtop(\gam,\gamz)$ is quotient.
\end{lemma}
\begin{proof}
Consider the following factorization of $\pi$:
\[\xymatrix{
\scrc(\epgam) \ar[d]_-{q}  & \mpgam \ar[l]_-{\mcd} \ar[d]^-{\pi} \\
\pi_0(\scrc(\epgam)) \ar[r]_-{\phi} & \piqtop(|\gam|,\gamz)
}\]where $\mcd$ is the decomposition morphism defined as follows. If $\alpha\in \mpgam(x,y)$, there are finitely many intervals $[a_1,b_1],...,[a_n,b_n]\subset \ui$ (ordered with respect to the ordering of $\ui$) such that $\alpha_i=\alpha_{[a_i,b_i]}$ is a non-trivial edge path. Now $\mcd(\alpha)$ is defined as the word $\alpha_1\alpha_2...\alpha_n$ in $\scrc(\epgam)$. If no restriction of $\alpha$ is a non-trivial edge path, then $\alpha$ is a null-homotopic loop based at some vertex $x\in \gamz$ and so we take $\mcd(\alpha)$ to be the identity $id_{x}$. Since $\alpha$ and $\alpha_1\cdot \alpha_2\cdot \dots\cdot \alpha_n$ are homotopic paths, the diagram commutes. The decomposition morphism is a direct generalization of the decomposition function in \cite[pp. 793]{Brazfretopgrp}; it is important to note that $\mcd$ is only a morphism of underlying algebraic graphs since it is not continuous on edge spaces.

For convenience, rename the sets $[0,a_1]$, $[b_1,a_{2}]$,..., $[b_{n-1},a_{n}]$, $[b_n,1]$ (some of which may be singletons) as $K_1,K_2,...,K_n,K_{n+1}$. Note that each restriction $\alpha_{K_i}$ is a trivial loop based at a vertex $x_i$ which has image in some vertex neighborhood $B(x_i,r_i)$. In particular, $x=x_1$ and $x_n=y$.

Given vertices $x,y\in \gamz$, suppose $U\subseteq \pi_{1}^{qtop}(\gam,\gamz)(x,y)$ such that $\phi^{-1}(U)$ is open in $\pi_0(\scrc(\epgam))(x,y)$. Thus $q^{-1}(\phi^{-1}(U))$ is open in $\scrc(\epgam)(x,y)$. Since $\pi$ is quotient, it suffices to show $\pi^{-1}(U)=\mcd^{-1}(q^{-1}(\phi^{-1}(U)))$ is open in $\mpgam(x,y)$. Suppose $\alpha\in \pi^{-1}(U)$ and note $\mcd(\alpha)$ lies in the open neighborhood $q^{-1}(\phi^{-1}(U))$. If $\mcd(\alpha)=id_x$ for some $x$, then the image of $\alpha$ lies in the contractible vertex neighborhood $B(x,1)$. The neighborhood $\{\beta\in \mpgam(x,x)|Im(\beta)\subseteq B(x,1)\}$ of $\alpha$ contains only null-homotopic loops and is therefore contained in $\pi^{-1}(U)$.

On the other hand, suppose $\mcd(\alpha)=\alpha_1....\alpha_n\in q^{-1}(\phi^{-1}(U))$ with decomposition as describe above. In particular, $\alpha_i\in \epgam(x_i,x_{i+1})$. We construct a neighborhood of $\alpha$ contained in $\pi^{-1}(U)$. Recall that $\mathcal{B}_i=\langle \ui,B(x_i,r_i)\rangle$ is a neighborhood of $\alpha_{K_i}$, and thus $\mathcal{B}_{i}^{K_i}$ is a neighborhood of $\alpha$ for each $i$. Since $\scrc(\epgam)$ is a $\spaces$-category, there are open neighborhoods $V_i$ of $\alpha_i$ in $\epgam(x_i,x_{i+1})$ such that the product $V_1V_2...V_n$ is contained in $q^{-1}(\phi^{-1}(U))$. By Lemma \ref{standardnbhdofedgepath}, there is a standard neighborhood $\mathcal{A}_i=\bigcap_{j=1}^{n_i}\left\langle \left[\frac{j-1}{n_i},\frac{j}{n_i}\right],U^{i}_{j}\right\rangle$, $n_i>2$ of $\alpha_i$ in $\mcp|\gam|$ such that $\mcu_i=\mathcal{A}_i\cap \epgam(x_i,x_{i+1})\subseteq V_i$ and such that $U^{i}_{1},U_{n_i}^{i}$ are vertex neighborhoods and $U^{i}_{2},...,U^{i}_{n-1}$ are edge neighborhoods. In particular, choose the $\mathcal{A}_i$ so that $U^{i}_{n_i}\cup U_{1}^{i+1}\subseteq B(x_{i+1},r_{i+1})$.

Now \[\mathcal{W}=\bigcap_{i=1}^{n}\mathcal{A}_{i}^{[a_i,b_i]}\cap \bigcap_{i=1}^{n+1}\mathcal{B}_{i}^{K_i}\cap \mpgam(x,y)\]is an open neighborhood of $\alpha$ in $\mpgam(x,y)$.

Suppose $\beta\in \mathcal{W}$. We clearly have $\beta(K_{i+1})\subset B(x_{i+1},r_i)$, however if $x_{i}=x_{i+2}$, it is possible that $\beta_{[a_i,b_{i+1}]}$ does not hit the vertex $x_{i+1}$. To deal with this possibility, we replace ``small" portions of $\beta$. For each $i=2,...,n$, let $s_i=L_{[a_{i-1},b_{i-1}]}\left(1-\frac{1}{n_{i-1}}\right)$ and $t_i=L_{[a_i,b_i]}\left(\frac{1}{n_i}\right)$ so that $K_i=[b_{i-1},a_i]\subset [s_i,t_i]$. Now define a path $\gamma$ to equal the path $\beta$ with the following exceptions: replace the portion of $\beta$ from $s_i$ to $b_{i-1}$ with the canonical arc from $\beta(s_i)$ to $x_i$, take $\gamma$ to be the constant at $x_i$ on $[b_{i-1},a_i]$, and replace the portion of $\beta$ from $a_i$ to $t_i$ with the canonical arc from $x_i$ to $\beta(t_i)$. Since $\gamma$ is given by changing $\beta$ only in contractible neighborhoods, $\gamma$ and $\beta$ are homotopic paths, i.e. $\pi(\beta)=\pi(\gamma)$. Moreover, $\gamma_i=\gamma_{[a_i,b_i]}$ is an edge path for each $i$ contained in $\mcu_{i}$. Thus \[\mcd(\gamma)=\gamma_1\gamma_2...\gamma_n\in \mcu_1\mcu_2...\mcu_n\subseteq V_1V_2...V_n\subseteq  q^{-1}(\phi^{-1}(U)).\]Finally, we see that \[\pi(\beta)=\pi(\gamma)=\phi(q(\mcd(\gamma)))\in U\] giving the inclusion $\mathcal{W}\subseteq \pi^{-1}(U)$.
\end{proof}
\begin{theorem}\label{computation} The fundamental $\spaces$-groupoid $\pit(\gam,\gamz)$ is naturally isomorphic to the free $\spaces$-groupoid $\mathscr{F}(\pi_0(\gam))$.
\end{theorem}
\begin{proof}
The embedding $\gam\to\mpgam$ given by $e\mapsto \alpha_e$ induces a $\spaces$-graph morphism $\pi_0(\gam)\to \pi_0(\mpgam)=\piqtop(\gam,\gamz)$. Additionally, the identity functor $\piqtop(\gam,\gamz)\to \pit(\gam,\gamz)$ is a morphism of $\qtop$-groupoids. The composition $\sigma:\pi_0(\gam)\to  \pit(\gam,\gamz)$ of these two morphisms is a morphism of $\spaces$-graphs which induces a morphism $\hat{\sigma}:\mathscr{F}(\pi_0(\gam))\to \pit(\gam,\gamz)$ of $\spaces$-groupoids. A straightforward generalization of \cite[3.14]{Braztopgrp} to $\spaces$-graphs with more than one vertex gives that $\hat{\sigma}$ is an isomorphism of the underlying groupoids. Therefore, it suffices to check the inverse $\hat{\sigma}^{-1}:\pit(\gam,\gamz)\to \mathscr{F}(\pi_0(\gam))$ is a $\spaces$-functor.

Consider the following commutative diagram. The upper horizontal functors are the $\qtop$-isomorphism from Lemma \ref{iso1} and the canonical $\qtop$-functor $\psi:\pi_0(\scrc^{\pm}(\gam))\to \scrc^{\pm}(\pi_0(\gam))$ from Lemma \ref{canonicalfunctor}. The vertical functor $R$ is the quotient $\qtop$-functor given by word reduction (See Remark \ref{freetopgroupoids}) and $\phi$ is the quotient $\qtop$ functor of Lemma \ref{quotient1}.
\[
\xymatrix{
\pi_0(\scrc(\epgam)) \ar[d]_-{\phi} \ar[r]^-{\cong} & \pi_0(\scrc^{\pm}(\gam))  \ar[r]^-{\psi} & \scrc^{\pm}(\pi_0(\gam)) \ar[d]^-{R}\\
\pi_{1}^{qtop}(\gam,\gamz) \ar[rr]_-{\hat{\sigma}^{-1}} & & \mathscr{F}_{R}(\pi_0(\gam))
}
\]
Since the top composition is a $\qtop$-functor and $\phi$ is quotient, $\hat{\sigma}^{-1}:\piqtop(\gam,\gamz)\to \mathscr{F}_{R}(\pi_0(\gam))$ is continuous on hom-spaces (by the universal property of quotient spaces) and is therefore a $\qtop$-functor. Applying the $\tau$-reflection gives that \[\hat{\sigma}^{-1}:\pit(\gam,\gamz)= \tau(\piqtop(\gam,\gamz))\to\tau(\mathscr{F}_{R}(\pi_0(\gam)))=\mathscr{F}(\pi_0(\gam))\] is a $\spaces$-functor.
\end{proof}

The topological group $\pit(\gam,\gamz)(v)$ at a vertex $v\in \gam$ is, by definition, the \textit{topological fundamental group} $\pi^{\tau}_{1}(|\gam|,v)$ of \cite{Braztopgrp,Brazsemi}. In light of Section \ref{vertexgroups}, we have the following Corollary.
\begin{corollary}\label{fundgroupoftopgraph}
The topological fundamental group $\pi^{\tau}_{1}(|\gam|,v)$ of a $\spaces$-graph $\gam$ is a free Graev topological group. In particular, if $\gam$ has more than one vertex and $T\subset \pi_0(\gam)$ is a maximal tree, then $\pi^{\tau}_{1}(|\gam|,v)\cong F_G(\pi_0(\gam)/T,\ast)$. If $\gam$ has a single vertex, then $\pi^{\tau}_{1}(|\gam|,v)\cong F_G(\pi_0(\gam)_+,\ast)\cong F_M(\pi_0(\gam))$.
\end{corollary}
\begin{corollary}
Every free $\spaces$-groupoid is the fundamental $\spaces$-groupoid of a $\spaces$-graph.
\end{corollary}
\begin{proof}
According to Remark \ref{everygraphisapcgraph}, a given $\spaces$-graph $\gam$ is isomorphic to $\pi_0(h(\gam))$ for some $\spaces$-graph $h(\gam)$. By Theorem \ref{computation}, $\pit(h(\gam),h(\gam)_0)\cong \mathscr{F}(\pi_0(h(\gam)))\cong \mathscr{F}(\gam)$.
\end{proof}
\section{Semicoverings and a proof of Theorem \ref{nielsonshreier}}\label{sectionsemicovering}
We recall the theory of semicovering spaces and apply it to $\spaces$-graphs. Our definitions and notation are those used in \cite{Brazsemi}. Given a space $X$ and point $x\in X$, let $(\mcp X)_x$ and $(\Phi X)_x$ be spaces of paths and homotopies (rel. endpoints) of paths starting at $x$ respectively with the compact-open topology. Let $\mcp X(x,x')$ be the subspace of $(\mcp X)_x$ consisting of paths ending at $x'$. In particular, $\mcp X(x,x)=\Omega(X,x)$ is the space of based loops.
\begin{definition}
\emph{
A \textit{semicovering} is a local homeomorphism $\pyx$ such that for each $y\in Y$, the induced maps $\mcp p:(\mcp Y)_y\to (\mcp X)_{p(x)}$ and $\Phi p:(\Phi Y)_y\to (\Phi X)_{p(y)}$ are homeomorphisms. The space $Y$ is a \textit{semicovering (space)} of $X$. If $\alpha$ is a path or homotopy of paths starting at $p(y)$, then $\tilde{\alpha}_{y}$ denotes the unique lift of $\alpha$ starting at $y\in Y$.
}
\end{definition}
Every covering (in the classical sense) is a semicovering, however, if $X$ does not have a simply connected cover (e.g. the Hawaiian earring), a semicovering of $X$ need not be a covering. Of particular importance to the current paper is the fact that a semicovering $\pyx$ induces an open covering morphism $\pit p:\pit Y\to \pit X$. In particular, for each $y_1,y_2\in Y$, $p(y_i)=x_i$, the induced map $p_{\ast}:\pi_{1}^{\tau}Y(y_1,y_2)\to \pit X(x_1,x_2)$ is an open embedding of spaces (of topological groups when $y_1=y_2$). Just as in classical covering theory, if $\beta$ is a loop based at $p(y)$, then $\tilde{\beta}_{y}$ is a loop based at $y$ if and only if $[\beta]$ lies in the image of the embedding $p_{\ast}$. Thus, since $\pi:\mcp X(x,x')\to \pit X(x_1,x_2)$ is continuous, $\{\beta|\tilde{\beta}_{y_1}(1)=y_2\}$ is an open subspace of $\mcp X(x_1,x_2)$.

Since $\spaces$-graphs can fail to be locally path connected, we cannot apply classical covering theory to study them. We use the fact that semicovering theory applies to certain non-locally path connected spaces called locally wep-connected spaces \cite[Definition 6.4]{Brazsemi}.

\begin{definition} \label{localpathconn} \emph{Let $X$ be a space.
\begin{enumerate}
\item A path $\alpha: \ui\to X$ is \textit{well-targeted} if for every open neighborhood $\mcu$ of $\alpha$ in $(\mcp X)_{\alpha(0)}$ there is an open neighborhood $V_1$ of $\alpha(1)$ such that for each $b\in V_1$, there is a path $\beta\in \mcu$ with $\beta(1)=y$.
\item A path $\alpha:\ui\to X$ is \textit{locally well-targeted} if for every open neighborhood $\mcu$ of $\alpha$ in $(\mcp X)_{\alpha(0)}$ there is an open neighborhood $V_1$ of $\alpha(1)$ such that for each $b\in V_1$, there is a well-targeted path $\beta\in \mcu$ with $\beta(1)=y$.
\end{enumerate}
See also the closely related definition of \textit{(locally) well-ended path} in \cite{Brazsemi}. A space $X$ is \textit{locally wep-connected} if for every pair of points $x,y\in X$, there is a locally well-targeted path from $x$ to $y$.}
\end{definition}
Every locally path connected space is locally wep-connected. There are also many spaces which are locally wep-connected but not locally path connected. For example, if $\gam$ is a $\spaces$-graph with a single vertex, the generalized wedge of circles $|\gam|=\Sigma(\gam_+)$ is locally wep-connected \cite[Proposition 6.7]{Brazsemi} but is only locally path connected if $\gam$ is locally path connected. The following Lemma generalizes this special case to arbitrary $\spaces$-graphs; the proof is nearly identical.
\begin{lemma}\label{lwepconnected}
Every $\spaces$-graph is locally wep-connected.
\end{lemma}
\begin{proof}
Since $|\gam|$ is locally path connected at each vertex, it suffices to find a locally well-targeted path from a vertex to each point $z\in |\gam|\backslash \gamz$. Suppose $z$ is the image of $(e,t)$ in $\gam(x,y)\times (0,1)$. Any path $\alpha:\ui\to |\gam|$ such that $\alpha(0)=x$, $\alpha(1)=z$, and having image on the edge $\{e\}\times [0,1)\subset |\gam|$ is locally well-targeted. The argument that $\alpha$ is locally well-targeted is identical to that in \cite[Proposition 6.7]{Brazsemi}.
\end{proof}
Since $\spaces$-graphs are locally wep-connected, we may apply semicovering theory to obtain the last ingredient for our proof of Theorem \ref{nielsonshreier}.
\begin{lemma}\label{embedding}
\cite[Corollary 7.20]{Brazsemi} If $\gam$ is a $\spaces$-graph, $x\in \gamz$ is a vertex, and $H$ is an open subgroup of the topological fundamental group $\pi_{1}^{\tau}(|\gam|,x)$, then there is a semicovering $p:Y\to |\gam|$, $p(y)=x$ such that the induced homomorphism $p_{\ast}:\pi_{1}^{\tau}(Y,y)\to \pi_{1}^{\tau}(|\gam|,x)$ is a topological embedding onto $H$.
\end{lemma}
\subsection{A semicovering of a $\spaces$-graph is a $\spaces$-graph}
The following result generalizes the fact that a covering (in the classical sense) of a graph is a graph and provides the last ingredient for a proof of Theorem \ref{nielsonshreier}.
\begin{theorem}\label{semicoveringofgraph}
A semicovering of a $\spaces$-graph is a $\spaces$-graph.
\end{theorem}
\begin{proof}
It suffices to assume the semicovering and $\spaces$-graph in question are connected. Let $p:Y\to |\gam|$ be a connected semicovering of $\spaces$-graph $\gam$. We find a $\spaces$-graph $\tgam$ such that $|\tgam|\cong Y$. Since $\gamz$ is a discrete subspace of $|\gam|$ and $p$ is a local homeomorphism, $p^{-1}(\gamz)$ is a discrete subspace of $Y$. Define the vertex space $\tgam_{0}=p^{-1}(\gamz)$. For $y_1,y_2\in \tgam_{0}$ such that $x_i=p(y_i)$ define \[\tgam(y_1,y_2)=\{e\in \gam(x_1,x_2)|\widetilde{(\alpha_e)}_{y_1}(1)=y_2\}\]with the subspace topology of $\gam(x_1,x_2)$.

Define a map $h:|\tgam|\to Y$ as follows: The restriction of $h$ to $\tgam_{0}$ is the identity. The map $h_{y_1,y_2}:\tgam(y_1,y_2)\times [0,1]\to Y$ given by $h_{y_1,y_2}(e,t)=\widetilde{(\alpha_e)}_{y_1}(t)$ is continuous since $\gam(x_1,x_2)\to \mcp |\gam|(x_1,x_2)$, $e\mapsto \alpha_e$ is continuous,  $\mcp p:(\mcp Y)_{y_1}\to (\mcp X)_{x_1}$ is a homeomorphism, and evaluation $\mcp Y\times [0,1]\to Y$, $(\beta,t)\mapsto \beta(t)$ is continuous. The maps $h_{y_1,y_2}$ induce the function $h$ on the image of $\tgam(y_1,y_2)\times [0,1]$ in $|\tgam|$. It follows from Lemma \ref{bijection} that $h$ is a bijection, Lemma \ref{continuous} that $h$ is continuous, and Lemma \ref{openmap} that $h$ is an open map. Therefore $h$ is a homeomorphism.
\end{proof}
To prove the following Lemmas, we make some observations about the edge spaces $\tgam(y_1,y_2)$. Let $A=\{\alpha_e\in\mcp |\gam|(x_1,x_2)|e\in \gam(x_1,x_2)\}$ and recall $B=\{\beta\in \mcp |\gam|(x_1,x_2)|\tilde{\beta}_{y_1}(1)=y_2\}$ is open in $ \mcp |\gam|(x_1,x_2)$. It is now clear that $A\cap B$ is open in $A$ and is the image of $\tgam(y_1,y_2)$ under the homeomorphism $\gam(x_1,x_2)\cong A$, $e\mapsto \alpha_e$. Therefore, $\tgam(y_1,y_2)$ is an open subspace of $\gam(x_1,x_2)$. It follows that whenever $p(y_1)=x_1$, 
\begin{itemize}
\item $\tgam_{y_1}= \gam_{x_1}$ and
\item $\gam(x_1,x_2)$ decomposes as the topological sum $\coprod_{p(y_2)=x_2}\tgam(y_1,y_2)$.
\end{itemize}
\begin{lemma}\label{bijection}
$h:|\tgam|\to Y$ is a bijection.
\end{lemma}
\begin{proof}
First, we show $h$ is surjective. It suffices to consider a point $y\in Y\backslash \tgam_0$. Note $p(y)$ is the image of a pair $(e,t)\in \gam(x_1,x_2)\times (0,1)$ in $|\gam|$. Fix $x_0\in \gamz$ and $y_0\in p^{-1}(x_0)$ and let $\beta$ be a path from $y_0$ to $y$ in $Y$. Since $|\gam|$ is a connected $\spaces$-graph, there is a sequence of vertices $x_0,x_1,...,x_{n-1},x_n$, edges $e_i\in\gam(x_{i-1},x_i)$, and $\delta_i\in \{\pm 1\}$ such that $p\beta$ is homotopic (rel. endpoints) to the concatenation \[\alpha=\alpha_{e_1}^{\delta_1}\cdot ...\cdot \alpha_{e_{n}}^{\delta_n}\cdot (\alpha_{e})_{[0,t]}.\]In particular, $\tilde{\alpha}_{y_0}(1)=y$. Let $y_1$ be the endpoint of the lift of $\alpha_{e_1}^{\epsilon_1}\cdot ... \cdot \alpha_{e_{n}}^{\delta_n}$ starting at $y_0$ and $y_2=\tilde{\alpha_e}(1)$. By our choice of $\alpha$, $p(y_i)=x_i$. Then the lift of $(\alpha_{e})_{[0,t]}$ starting at $y_1$ ends at $h(e,t)=\widetilde{(\alpha_e)}_{y_1}(t)=y$.

For injectivity, suppose $z,z'\in |\tgam|$. If one of $z$ or $z'$ is a vertex and $h(z)=h(z')$, then $z,z'\in \tgam_0$ and it follows that $z=z'$. Therefore it suffices to check that $h$ is injective on $|\tgam|\backslash \tgam_0$. Suppose $z$ is the image of $(e,t)\in \tgam(y_1,y_2)\times (0,1)$ and $z'$ is the image of $(f,u)\in \tgam(y_3,y_4)\times (0,1)$ under $h$. If $h(z)=\widetilde{(\alpha_e)}_{y_1}(t)=\widetilde{(\alpha_f)}_{y_3}(u)=h(z')$, then $\alpha_e(t)=\alpha_f(u)$ in $|\gam|$, however, this only occurs if $e=f$ and $t=u$ and thus $z=z'$.
\end{proof}
\begin{lemma}\label{continuous}
$h:|\tgam|\to Y$ is continuous.
\end{lemma}
\begin{proof}
Since the maps $h_{y_1,y_2}$ are continuous, $h$ is continuous if $|\gam|$ is given the quotient topology. However, we using the coarser topology on $|\gam|$ (which may differ from the quotient topology at vertices). Thus it remains to check that $h$ is continuous at each vertex of $|\tgam|$. Suppose $y_0\in \tgam_0$ and $V$ is an open neighborhood of $h(y_0)=y_0$ in $Y$ such that $p$ maps $V$ homeomorphically onto a vertex neighborhood $B(x_0,r)$ of $p(y_0)=x_0$. Since $h$ is a bijection (Lemma \ref{bijection}), \[h(\widetilde{B}(y_0,r))=\{\widetilde{(\alpha_e)}_{y_0}(t)\in Y|0\leq t<r, e\in \gam_{y_0}\}\cup \{\widetilde{(\overline{\alpha_f})}_{y_0}(t)\in Y|0\leq t<r, f\in \gam^{y_0}\}.\]But $p$ maps $V$ homeomorphically onto the path connected set $B(x_0,r)$. Therefore, the lifts of all paths in $B(x_0,r)$ starting at $y_0$ have image in $V$. It follows that $h(\widetilde{B}(y_0,r))= V$.
\end{proof}
\begin{lemma}\label{openmap}
$h:|\tgam|\to Y$ is an open map.
\end{lemma}
\begin{proof}
Since $|\gam|$ is locally wep-connected (Lemma \ref{lwepconnected}), $Y$ is locally wep connected \cite[Corollary 6.12]{Brazsemi}. Thus for each $y\in Y$, evaluation $ev_1:(\mcp Y)_{y}\to Y$, $ev_1(\beta)=\beta(1)$ is quotient \cite[Proposition 6.2]{Brazfretopgrp}. Additionally, if $p(y)=x$, $\mcp p:(\mcp Y)_{y}\to (\mcp |\gam|)_{x}$ has a continuous inverse $L:(\mcp |\gam|)_{x}\to (\mcp Y)_{y}$. To show $h$ is open, we use the fact that the composition $ev_1 L:(\mcp |\gam|)_{x}\to Y$, $\beta\mapsto \tilde{\beta}_{y}(1)$ is quotient whenever $p(y)=x$.

Fix a vertex $y_0\in \tgam_0$ and $p(y_0)=x_0$. Suppose $\tilde{B}(y_0,r)$, $0<r<1$ is a vertex neighborhood of $y_0\in \tgam_0$. As in the previous lemma, if \[V=\{\widetilde{(\alpha_e)}_{y_0}(t)\in Y|0\leq t<r, e\in \gam_{y_0}\}\cup \{\widetilde{(\overline{\alpha_f})}_{y_0}(t)\in Y|0\leq t<r, f\in \gam^{y_0}\},\] then $h(\tilde{B}(y_0,r))=V$. Again, we use the fact that $p(V)=B(x_0,r)$ and if $\gamma:\ui\to B(x_0,r)$ is the canonical arc from $x_0$ to a given point $z\in B(x_0,r)$, then $\tilde{\gamma}_{y_0}$ has image in $V$. We claim that $V$ is open in $Y$.

Since $ev_1 L$ is quotient, it suffices to show $L^{-1}(ev_{1}^{-1}(V))$ is open in $(\mcp|\gam|)_{x_0}$. If $\beta\in L^{-1}(ev_{1}^{-1}(V))$, then $\tilde{\beta}_{y_0}(1)\in V$ and thus $\beta(1)\in B(x_0,r)$. Let $\gamma$ be the canonical arc from $x_0$ to $\beta(1)$ in $B(x_0,r)$ and recall $Im(\tilde{\gamma}_{y_0})\subset V$. Thus $\widetilde{\beta\cdot\overline{\gamma}}_{y_0}$ is a loop based at $y_0$ and $[\beta\cdot\overline{\gamma}]$ lies in the open subgroup $p_{\ast}(\pit(Y,y_0))$ of $\pit(|\gam|,x_0)$. Since $\pi:\Omega(X,x_0)\to \pit(X,x_0)$ is continuous, there is a basic open neighborhood $\mcu=\bigcap_{j=1}^{n}\left\langle \left[\frac{j-1}{n},\frac{j}{n}\right],U_j\right\rangle$ of $\beta\cdot \overline{\gamma}$ in $\mcp|\gam|$ such that
$\mcu\cap \Omega(X,x_0)\subseteq \pi^{-1}(p_{\ast}(\pit(Y,y_0)))$. Since $B(x_0,r)$ is contractible, we may assume 1. $n$ is even, 2. $U_1=B(x_0,r)$, and 3. $U_k=B(x_0,r)$ for $k\geq n/2$. Now $\mcv=\mcu_{[0,1/2]} \cap (\mcp|\gam|)_{x_0}$ is an open neighborhood of $\beta$ in $(\mcp|\gam|)_{x_0}$ which we claim is a subset of $L^{-1}(ev_{1}^{-1}(V))$. Suppose $\beta '\in \mcv$. Then $\beta '(1)\in U_{n/2}=B(x_0,r)$ and, if $\gamma '$ is the canonical arc from $x_0$ to $\beta '(1)$, then $\beta\cdot\overline{\gamma '}\in \mcu\cap \Omega(X,x_0)\subseteq \pi^{-1}(p_{\ast}(\pit(Y,y_0)))$. Thus $\widetilde{(\beta\cdot\overline{\gamma '})}_{y_0}(1)=y_0$. Since $\widetilde{\gamma '}_{y_0}$ has image in $V$, \[ev_1L(\beta ')=\tilde{\beta '}_{y_0}(1)=\tilde{\gamma '}_{y_0}(1)\in V.\] It follows that $L^{-1}(ev_{1}^{-1}(V))$ is open in $(\mcp|\gam|)_{x_0}$.

Let $y_1,y_2\in \tgam_0$, $p(y_i)=x_i$, and suppose $U\times (a,b)\subset \tgam(y_1,y_2)\times (0,1)$ is an edge neighborhood in $|\tgam|$. Note that \[W=h(U\times (a,b))=\{\widetilde{(\alpha_e)}_{y_1}(t)\in Y|(e,t)\in \tgam(y_1,y_2)\times (0,1)\}.\] We show $L^{-1}(ev_{1}^{-1}(W))$ is open in $(\mcp X)_{x_1}$.

Recall $\tgam(y_1,y_2)$ is defined to be an open subspace of $\gam(x_1,x_2)$. Therefore $ph$ maps $U\times (a,b)$ homeomorphically onto the corresponding edge neighborhood $U\times (a,b)\subset |\gam|$. If $\beta\in L^{-1}(ev_{1}^{-1}(W))$, then $\widetilde{\beta}_{y_1}(1)=\widetilde{(\alpha_e)}_{y_1}(t)\in W$ for some $(e,t)\in U\times (a,b)$. Let $\gamma=(\alpha_e)_{[0,t]}$. Since $\widetilde{\beta\cdot \overline{\gamma}}_{y_1}$ is a loop based at $y_1$, we have, as in the vertex neighborhood case, that $\beta\cdot \overline{\gamma}$ lies in the open neighborhood $\pi^{-1}(p_{\ast}(\pi_1(|\gam|,x_1)))\subset \Omega(|\gam|,x_1)$. Take a basic open neighborhood $\mcu=\bigcap_{j=1}^{n}\left\langle \left[\frac{j-1}{n},\frac{j}{n}\right],U_j\right\rangle$ of $\beta\cdot \overline{\gamma}$ in $\mcp|\gam|$ such that
$\mcu\cap \Omega(|\gam|,x_1)\subseteq \pi^{-1}(p_{\ast}(\pit(|\gam|,x_1)))$. In particular, we may assume 1. $n$ is even, 2. $U_1=U_n=B(x_1,r)$ is a vertex neighborhood, 3. when $n/2\leq k<n$, $U_k$ is an edge neighborhood of the form $A\times (r_k,s_k)$ for an open set $A\subseteq U$, and 4. $(r_{n/2},s_{n/2})=(r_{(n/2)+1},s_{(n/2)+1})\subseteq (a,b)$. Now $\mcv=\mcu_{[0,1/2]}\cap (\mcp |\gam|)_{x_1}$ is an open neighborhood of $\beta$ in $(\mcp |\gam|)_{x_1}$. Note that $\beta(1)\in A\times(r_{n/2},s_{n/2})$.

We check that $\mcv\subseteq L^{-1}(ev_{1}^{-1}(W))$. If $\beta '\in \mcv$, then $\beta '(1)\in A\times (r_{n/2},s_{n/2})\subseteq U\times (a,b)\subset |\gam|$. If $\beta '(1)$ is the image of $(f,u)\in A\times (r_{n/2},s_{n/2})\subseteq \gam(x_1,x_2)\times (0,1)$ in $|\gam|$, then $\alpha_{f}(u)=\beta '(1)$. Choose any $\epsilon$ such that $\frac{n-2}{n}<\epsilon <1$. Define a path $\gamma '$ so that
\begin{itemize}
\item $\gamma '(v)=(\alpha_f)_{[0,t]}(v)$ for $v\in [0,\epsilon]$.
\item $(\gamma ')_{[\epsilon ,1]}$ is the canonical arc from $z=(\alpha_f)_{[0,t]}(\epsilon)$ to $\alpha_f(u)$ in $A\times (r_{n/2},s_{n/2})$.
\end{itemize}
Notice that $\gamma '$ is constructed so that 1. $\beta '(1)=\gamma '(1)$, 2. $\beta '\cdot \overline{\gamma '}$ is a loop in $\mcu\subseteq \pi^{-1}(p_{\ast}(\pi_1(|\gam|,x_1)))$ based at $x_1$, and 3. $\gamma '$ is homotopic (rel. endpoints) to $(\alpha_f)_{[0,u]}$ (consequently $\widetilde{\gamma '}_{y_1}(1)=\widetilde{(\alpha_f)}_{y_1}(u)$). Since $[\beta ' \cdot \overline{\gamma '}]\in p_{\ast}(\pi_1(|\gam|,x_1))$, we have  \[ev_1L(\beta)=\widetilde{(\beta ')}_{y_1}(1)=\widetilde{\gamma '}_{y_1}(1)=\widetilde{(\alpha_f)}_{y_1}(u)\in W. \] Thus $\mcv\subseteq L^{-1}(ev_{1}^{-1}(W))$.
\end{proof}
\subsection{A proof of Theorem \ref{nielsonshreier}}
We conclude with a proof of Theorem \ref{nielsonshreier}.
\begin{proof} Suppose $X$ is a space with basepoint $\ast\in X$ and $H$ is an open subgroup of the free topological group $F_G(X,\ast)$. Let $h(X)$ be a space such that $\pi_0(h(X))=X$ (see Remark \ref{everygraphisapcgraph}) and $\gam$ be the $\spaces$ graph with $\gamz=\{a,b\}$ (i.e. two vertices), $\gam(a,b)=h(X)$, and $\gam(b,a)=\emptyset$. Note that the edge space of $\pi_0(\gam)$ is precisely $X$. By Theorem \ref{computation}, $\pit(\gam,\gamz)$ is isomorphic to the free $\spaces$-groupoid $\mathscr{F}(\pi_0(\gam))$. A tree $T\subseteq \pi_0(\gam)$ is given by taking $T_0=\{a,b\}$ with edge space $T=\{\ast\}$. Note $\pi_0(\gam)/T\cong X$ as based spaces. Theorem \ref{vertexgroup} gives the middle isomorphism in
\[\pi_{1}^{\tau}(|\gam|,a)=\pit(\gam,\gamz)(a)\cong F_G(\pi_0(\gam)/T,\ast)\cong F_G(X,\ast).\]
By Lemma \ref{embedding}, there is a semicovering $p:Y\to |\gam|$, $p(y)=a$ such that the induced homomorphism $\pi_{1}^{\tau}(Y,y)\to \pi_{1}^{\tau}(|\gam|,a)\cong F_G(X,\ast)$ is a topological embedding onto $H$. According to Theorem \ref{semicoveringofgraph}, the semicovering space $Y$ is a $\spaces$-graph. Finally, Corollary \ref{fundgroupoftopgraph} applies to $Y$ to give that $\pi_{1}^{\tau}(Y,y)\cong H$ is a free Graev topological group.
\end{proof}

\begin{thebibliography}{99}


\bibitem{AT} Arhangel'skii, Tkachenko, \emph{Topological Groups and Related Structures}. Atlantis Studies in Mathematics, 2008.

\bibitem{Brazfretopgrp}
J.~Brazas, \emph{The topological fundamental group and free topological groups},
  Topology Appl. 158 (2011) 779--802.

\bibitem{Braztopgrp}
J.~Brazas, \emph{The fundamental group as a topological group}, Preprint,
  	arXiv:1009.3972v5, 2012.

\bibitem{Brazsemi}  J.~Brazas, \emph{Semicoverings: a generalization of covering space theory}, Homology Homotopy Appl. 14 no. 1 (2012), 33--63.

\bibitem{Brown06}
R.~Brown, \emph{Topology and Groupoids}, Booksurge PLC, 2006.


\bibitem{BHsubgroup}
R.~Brown, J.P.L.~Hardy, \emph{Subgroups of free topological groups and free topological products of topological groups}, J. London Math. Soc. (2) 10 (1975), 431--440.

\bibitem{Brown75}
R.~Brown, \emph{Some non-projective subgroups of free topological groups}, Proc. Amer. Math. Soc. 52 (1975) 443--440.

\bibitem{Clarke}
F.~Clarke, \emph{The commutator subgroup of a free topological group need not be projective}, Proc. Amer. Math. Soc. 57 no. 2 (1976) 354--356.

\bibitem{Graev} Graev, M.I. \emph{Free topological groups}. Amer. Math. Soc. Transl. 8 (1962) 305-365.

\bibitem{Har80}
D.~Harris, \emph{Every space is a path component space}, Pacific J. Math. 91 (1980)
  95--104.

\bibitem{Hatcher02}
A.~Hatcher, Algebraic Topology, Cambridge University Press, 2002.

\bibitem{Hig71}
P.~Higgins, \emph{Notes on categories and groupoids}, Vol.~32, Van Nostrand Reinhold
  Co. London, 1971, also: Reprints in Theory Appl. Categories No. 7 (2005).

\bibitem{HuntMorris}
D.C.~Hunt, S.A.~Morris, \emph{Free subgroups of free topological groups}, Proc. Second
Internat. Conf. Theory of Groups, Canberra, Lecture Notes in Mathematics 372 (Springer,
Berlin, 1974), 377--387.

\bibitem{KellyEnriched}
G.~Kelly, \emph{Basic concepts of enriched category theory}, Vol.~64 of London Math. Soc. Lec. Notes Series, Cambridge University Press, 1982, also: Reprints in Theory Appl. Categories 10 (2005).

\bibitem{Markov} Markov, A.A. \emph{On free topological groups}. Izv. Akad. Nauk. SSSR Ser. Mat. 9 (1945) 3-64 (in Russian); English Transl.: Amer. Math. Soc. Transl. 30 (1950) 11-88; Reprint: Amer. Math. Soc. Transl. 8 (1) (1962) 195-272.

\bibitem{Nickolas75}
P.~Nickolas, \emph{A Schreier theorem for free topological groups}, Bulletin of the Australian Math. Soc. 13 (1975) 121--127.

\bibitem{Nickolas76}
P.~Nickolas, \emph{Subgroups of the free topological group on [0,1]},
  J. London Math. Soc. (2) 12 (1976) 199--205.

\bibitem{Po91}
H.~Porst, \emph{On the existence and structure of free topological groups}, Category Theory at Work (1991) 165--176.

\bibitem{Schreier} O.~Schreier, \emph{Die Untergruppen der freien Gruppen}, Abh. Mat Sem. Univ. Hamburg 3 (1927) 167–-169.

\bibitem{Sipacheva} O.V.~Sipacheva, \emph{The Topology of Free Topological Groups}. J. Math. Sci. Vol. 131, No. 4, (2005), 5765-–5838.

\bibitem{Thomas} B.V.S.~Thomas, \emph{Free topological groups}. General Topology and its Appl. 4 (1974) 51--72.

\end{thebibliography}
\end{document}